\documentclass[12pt]{amsart}

\usepackage{ljm-auth}
\usepackage{amsfonts,amsthm,latexsym,amsmath,amssymb,dsfont}
\usepackage{amssymb}
\usepackage{graphicx}
\usepackage{fancyhdr}
\newtheorem{ques}{Question}

\newtheorem{dfn}{Definition}
\newtheorem{fct}{Fact}

\newtheorem{exl}{Example}
\newtheorem{prpsn}{Proposition}

\newtheorem{thrm}{Theorem}

\newtheorem{crlr}{Corollary}

\newtheorem{lmm}{Lemma}

\theoremstyle{plain}
\theoremstyle{definition}
\renewcommand{\iff}{\Leftrightarrow}
\renewcommand{\phi}{\varphi}

\author{Alex Gavryushkin}
\crauthor{A.\,Gavryushkin} 

\tit{Decidable models of small theories}
\shorttit{Decidable models of small theories} 

\begin{document}

\maketit

\begin{center} Department of Computer Science, The University of Auckland,\\
Private Bag 92019, Auckland, New Zealand\\
email: a.gavruskin@auckland.ac.nz \end{center}

\abstract{Many counterexamples are known in the class of small theories due to
Goncharov~\cite{Goncharov} and Millar~\cite{Millar}. The prime model of a decidable small theory is not 
necessarily decidable. The saturated model of a hereditarily decidable small theory
is not necessarily decidable. A homogeneous model with uniformly decidable type
spectra is not necessarily decidable. In this paper, I consider the questions of 
what model theoretic properties are sufficient for the existence of such 
counterexamples. I introduce a subclass of the class of small theories,
which I call AL theories, show the absence of Goncharov-Millar counterexamples
in this class, and isolate a model theoretic property that implies the
existence of such anomalies among computable models.} \notes{0}{

\subclass{03C57, 03D75}%
\keywords{theories, types, decidable theories, AL theories, computable and decidable
models}%
\thank{This work was partially funded by the subsidy allocated to Kazan
Federal University for the state assignment in the sphere of
scientific activities}}


I consider only countable structures of countable languages. And only {\em small} theories.

\begin{dfn}
\textup{A first-order theory $T$ is {\em small} if the set of finite first-order types of $T$ without parameters, $S(T)$, is countable.
}
\end{dfn}

Throughout the paper all theories are assumed to be small. 

\begin{fct}  Every small theory has a prime model and a saturated model. 
These models are unique up to isomorphism. 
\end{fct}
\begin{fct}\label{1} If $p$ is a type of a small theory $T$ 
and $A\models p(\bar a)$ then the theory 
$Th(A,\bar a)$ has a prime model $(A_{\bar a},\bar c)$. The isomorphism 
type of the structure $A_{\bar a}$ does not depend on the choice of $A$
and $\bar a$, it only depends on the type $p$. Since \textbf{we consider 
structures up to isomorphism}, denote the structure by $A_{p}$.
\end{fct}

\begin{dfn}
\textup{Call the structure $A_{p}$ from Fact \ref{1} {\em $p$-prime}, or {\em almost prime} if the type is not specified.
}
\end{dfn}

The set $\mathcal{AP}_{T}$ of all almost prime models of a theory $T$ is preordered 
by the relation $\preceq$ of elementary embeddability. The preorder induces 
a partial order on the factor-set 
$\mathcal{AP}_{T}/\sim$, where $A\sim B\iff (A\preceq B~\&~B\preceq A)$, in the
natural way. Note that
$(\mathcal{AP}_{T}/\sim,\preceq)$ has a unique least element---the prime model of $T$.

\begin{dfn}
\textup{We call the partial order $(\mathcal{AP}_{T}/\sim,\preceq)$ the {\em fundamental order} of the theory $T$.
}
\end{dfn}

The notion of fundamental order is also known as Rudin--Keisler order~\cite{Sud, Sud2}. In this paper, I follow the terminology of Lascar--Poizat~\cite{lp} and Baldwin--Berman~\cite{bb}, who introduced and studied fundamental orders of theories in general. 

To illustrate the definition, let us consider the following 
simple though important examples of fundamental orders. 

\begin{exl}\label{exlALT}
The following theories $T$ are AL theories in the sense of Definition~\ref{dfnALT} below. 
\begin{enumerate}
\item A saturated structure is almost prime if and only if the structure is $\aleph_{0}$-categorical. For the theory $T$ of such a structure, $\mathcal{AP}_{T}$ is the
one-element partial ordering.  
\item If a theory $T$ is $\aleph_{1}$- but not $\aleph_{0}$-categorical then $\mathcal{AP}_{T}\cong\omega$.
\item If a complete theory $T$ has finitely many countable models then $\mathcal{AP}_{T}$ has a greatest element.
\end{enumerate}
\end{exl}
\begin{proof}
(1) is obvious. For (2) we note that the saturated model of such a theory is not almost prime
and the rest of models are almost prime. 

(3) A type $p$ of a theory $T$ is called powerful if every model of $T$ realising $p$ realises
every type of $T$ as well. If $T$ is a complete theory with finitely many, but more than one, 
countable models then it has a non-principal powerful type $p$.  A $p$-prime structure gives 
the greatest element of $\mathcal{AP}_{T}$.
\end{proof}

\begin{prpsn}[\cite{Sud}]
If $A_{p}\sim A_{q}$ but $A_{p}\not\cong A_{q}$, then there is a structure $A$ such that $A\sim A_{p}$ but $A$ is not almost prime.
\end{prpsn}

\proof Form an elementary chain $A_{0} \subseteq  A_{1} \subseteq \ldots$ where $A_{n} \cong A_{p}$ if $n$ is even and $A_{n} \cong A_{q}$ if $n$ is odd. Put $A=\bigcup\limits_{n\in\omega} A_{n}$.
\endproof

The structure $A$ can be presented as a union of an elementary chain of isomorphic almost 
prime structures, but $A$ itself is not almost prime. 
Following Sudoplatov~\cite{Sud, Sud2}, I shall call such a structure limit:

\begin{dfn}
\textup{A structure is {\em $p$-limit} if it is a union of an elementary chain 
of $p$-prime structures but it itself is not $p$-prime. 
A structure is {\em limit} if it is a union of an elementary chain of isomorphic almost prime
structures but it itself is not almost prime.}
\end{dfn}

\begin{dfn}\label{dfnALT}
\textup{A complete small theory $T$ is an {\em AL theory} if every countable non-saturated model of $T$ is either almost prime or limit.}
\end{dfn}

Definition~\ref{dfnALT} can be rewritten syntactically~\cite{Sud}.

Example~\ref{exlALT} above gives some well-known examples of AL theories. 

Note that a saturated structure is limit if and only if its theory has a non-principal powerful type, i.\,e. $\mathcal{AP}_{T}$ has a maximal element.

Let $\mathcal{LS}_{T}$ denote a set of all limit models of an AL theory $T$.
The structure of the spectrum of models of the theory $T$ is determined by a pre-ordering $
\mathcal{AP}_{T}$ and a function $\lambda_{T}:\mathcal{AP}_{T}\to2^{\mathcal{LS}_{T}}$ mapping a $p$-prime structure to the set of all $p$-limit structures~\cite{Sud,Sud2}. Think of $\lambda_{T}$ as of a disjoint union of bipartite graphs. $\mathcal{LS}_{T}=\bigcup\limits_{M\in
\mathcal{AP}_{T}}\lambda_{T}(M)$.

\begin{ques}
How distinct are the class of AL theories and the class of small theories?
\end{ques}

Naturally, the class of AL theories is a proper subclass of the class of small theories. 
The following definition isolates a model theoretic property of small theories
that might be violated in the class of AL theories. Sudoplatov proved in~\cite{Sud} that
if we relax Definition~\ref{dfnALT} using Definition~\ref{dfnWL} then we get the class
of all small theories. 

\begin{dfn}[\cite{Sud2}]\label{dfnWL}
{\em A structure is {\em weakly limit} if it is the union of an elementary chain of almost prime 
structures. }
\end{dfn}

For the sake of completeness, I include the proof of Sudoplatov's lemma~\cite{Sud} that every countable model of a small theory is either almost prime or weakly limit. 

\begin{lmm}[S. Sudoplatov \cite{Sud}]
Every countable model of a small theory is either almost prime or weakly limit. \qed
\end{lmm}

\begin{proof}
Let $A$ be a countable model of a small theory $T$. Let $a_0,a_1,\ldots$ be an enumeration 
of the domain of $A$. Consider the type $tp(a_0)$ of $a_0$ in $A$. Since $T$ is small, there 
is a prime model $(A_0,a_0)$ of $tp(a_0)$ in the signature enriched by one constant. We can 
choose $A_0$ to be an elementary submodel of $A$. If $A_0$ is isomorphic to $A$, we are 
done and $A$ is almost prime. If not, take the first element $b_1$ in our enumeration 
$a_0,a_1,\ldots$ such that $b_1\in A\setminus A_0$. Consider a prime model $
(A_1,a_0,b_1)$ of $tp(a_0,b_1)$. Again, if $A_1$ is isomorphic to $A$ then $A$ is almost 
prime. If not, we continue the process. If the process terminates after finitely many steps, the 
model $A$ is prime over $(a_0,b_1,\ldots,b_n)$ for some $n$ and hence is almost prime. If 
not, we obtain an elementary chain $A_0\preceq A_1\preceq\ldots$ of almost prime 
structures. Since $a_0,a_1,\ldots$ is an enumeration of all elements of $A$, we have that $
\cup_iA_i=A$ and $A$ is weakly limit.
\end{proof}

The following general theorem shows that the class of AL theories is rather different from
the class of all small theories if we study decidable models. This theorem
is the key property to the absence of Goncharov-Millar counterexamples in the class
of AL theories. 

\begin{thrm}\label{thrmIdeal}
Let $T$ be an AL theory. Then the set of decidable almost prime models of $T$ 
forms an ideal in the fundamental order of $T$.
\end{thrm}

\proof
Let us first note the following property of decidable models of an AL theory, which is
interesting on its own. 

\begin{lmm}\label{APdecidable}
Let $T$ be an AL theory and $A$ its non-saturated decidable model. Then there
exist a type $p$ and a decidable $p$-prime model $A_p$ such that the type
spectra of the structures $A$ and $A_p$ coincide.  
\end{lmm}

\proof
If $A$ is almost prime, there is nothing to prove. Suppose $A$ is $p$-limit for some $p$. 
Since $A$ is decidable, the set of types realised in $A$ is uniformly computable. 
This gives a uniformly computable enumeration of principal types realised in 
$(A_{\bar a},\bar c)$, where $\bar a$ is a realisation of $p$ in $A$. 
Hence, $(A_{\bar a},\bar c)$ is decidable and so is $A_{\bar a}$. It remains
to note that the type spectra of $A$ and $A_{\bar a}$ coincide. 
\endproof

I now prove that the set of decidable models of an AL theory
is directed upwards. 

\begin{lmm}
Let $T$ be an AL theory and $A$ and $B$ be its decidable models. 
Then there exists a decidable model $C$ such that $C\succeq A$
and $C\succeq B$.
\end{lmm}
\proof
Using Lemma~\ref{APdecidable} above, one can assume that $A$ is a $p$-prime and $B$
is a $q$-prime structure for some decidable types $p$ and $q$. Take $C$
to be a decidable structure that realises both $p$ and $q$. Such a structure 
exists because otherwise the theory $T$ would have more than countably
many types. 
\endproof

To finish the proof of the theorem, we need to show that the set of decidable 
almost prime models of $T$
is closed downwards. Let $A_p$ be a decidable almost prime model of $T$ 
and $A_q$ be an almost prime model of $T$ such that $A_q\preceq A_p$. 
We want to show that $A_q$ is decidable. Since $T$ is AL, 
there are only finitely many almost prime structures with pairwise 
distinct type spectra $A_0=A_q,A_1,\ldots,A_n=A_p$ such that 
$A_0\preceq A_1\preceq\ldots\preceq A_n$. Using Millar's type omitting
theorem~\cite{Millar} (see also~\cite{Goncharov}), we can realise all the
types that are realised in $A_q$ and omit all the types that are not realised
in $A_q$, in a decidable structure $B$. That is, the structure $B$ has the same
type spectra as $A_q$ and is decidable. We apply Lemma~\ref{APdecidable} to 
prove that $A_q$ is itself decidable. Since the set of decidable almost prime 
models of the theory $T$ is directed upwards and closed downwards, it
forms an ideal. 
\endproof

\begin{crlr}\label{crlrPrimeDec}
If an AL theory $T$ is decidable then $T$ has a decidable prime model.
\end{crlr}

\proof
Since the theory $T$ is decidable, it has a decidable model $A$. If the model 
$A$ is saturated then the prime model of $T$ is decidable, see~\cite{Goncharov}. 
If the model $A$ is not saturated then it is $p$-prime or $p$-limit
and we apply Theorem~\ref{thrmIdeal} to prove 
that the prime model of $T$ is decidable. 
\endproof

Note that Corollary~\ref{crlrPrimeDec} does not hold in the class of all small 
theories. In~\cite{gavruskin}, we construct a decidable small theory $T$ whose all
types are decidable yet whose prime model is not decidable. Millar~\cite{Millar}
was the first to construct such an example but his construction uses an infinite 
language in an essential way, while our structure is a graph. Since the prime model 
of a theory with decidable saturated model is decidable, we have an example of
a decidable small theory whose saturated model is not decidable. 

Although the class of AL theories and the class of small theories have similar 
model theoretic properties, they are significantly different from the computability
point of view. Corollary~\ref{crlrPrimeDec} along with the following corollary 
demonstrate this difference. 

\begin{crlr}
If $T$ is a decidable AL theory whose types are all decidable, then every almost prime model 
of $T$ is decidable.
\end{crlr}

\proof
Let $A_p$ be an almost prime model. Since $p$ is a decidable type, it can be realised 
is a decidable model. Applying Theorem~\ref{thrmIdeal} we get that $A_p$ is decidable. 
\endproof




In the light of Theorem~\ref{thrmIdeal}, a natural question to ask would be what ideals can be 
formed by decidable models? 
The following theorem of ours answers this question in full for the class of theories with
lattice-like fundamental orders. The question remains open in the general case of an arbitrary
fundamental order. 

Let $\mathcal{AP}^{\mathcal D}_{T}$ be the ideal 
formed by decidable models of an AL theory $T$. 

\begin{thrm}[A.\,Gavryushkin and B.\,Khoussainov~\cite{gavruskin}]
Let $\mathcal L$ be a finite lattice and $\mathcal L'$ be its ideal. Then there exists an AL
theory $T$ such that:
\begin{enumerate}
\item The fundamental order of $T$ is $\mathcal L$, that is, $(\mathcal{AP}_{T}/\sim,\preceq)\cong\mathcal L$.
\item The spectra of decibel models of $T$ is $\mathcal L'$, that is, $\mathcal{AP}^{\mathcal D}_{T}\cong\mathcal L'$.
\end{enumerate}
\end{thrm}

Computable models behave rather differently than decidable ones. I conclude with
a theorem that shows that computable models do not form an ideal in the class of all
models of an AL theory. 

\begin{thrm}[A.\,Gavryushkin and B.\,Khoussainov~\cite{gavruskin}]
For every finite lattice $\mathcal L$, there exists a theory $T$ of finite signature with countably many models such that:
\begin{enumerate}
\item The fundamental order of $T$ without the least element is isomorphic to $\mathcal L$.
\item For all $p\in\mathcal L$, the class of models corresponding to $p$ contains infinitely many models of which exactly one is computable.
\end{enumerate}
\end{thrm}

\end{document}